\documentclass[12pt]{article}

\usepackage[margin=1.17in]{geometry}
\usepackage{amssymb,amsmath}
\usepackage{amsthm,enumitem,mathtools,microtype}
\usepackage{hyperref}
\usepackage{mathrsfs}
\usepackage{textcomp}
\usepackage{xcolor}
\hypersetup{colorlinks, citecolor=black, filecolor=black, linkcolor=black, urlcolor=black, hypertexnames=false, pdftitle={A Counterexample to Interior Lp Estimates for Nonlocal Elliptic Equations with Continuous Kernels}, pdfauthor={Tianling Jin, Chang Liu, Jingang Xiong}}
\usepackage[abbrev,nobysame]{amsrefs}
\usepackage{todonotes}
\usepackage{verbatim}
\usepackage{sectsty}
\usepackage[T1]{fontenc}
\usepackage{lmodern}
\usepackage{needspace}

\DeclareMathOperator{\supp}{supp}
\DeclarePairedDelimiter{\norm}{\lVert}{\rVert}

\numberwithin{equation}{section}
\newtheorem{theorem}{Theorem}[section]
\newtheorem{lemma}[theorem]{Lemma}
\theoremstyle{remark}

\begin{document}
\title{\textbf{Failure of interior $L^p$ estimates for nonlocal elliptic equations with continuous kernels}\bigskip}
\author{ \ \ Chang Liu}
\date{}
\maketitle

\begin{abstract}
Let $0<\sigma<2$ and $1\leq p<\infty$. We construct a fixed linear nonlocal elliptic operator $L$ and functions $u_m\in C^\infty(\mathbb R^n)\cap L^\infty(\mathbb R^n)$ such that
$$
\|u_m\|_{L^\infty(\mathbb R^n)}+\|Lu_m\|_{L^p(B_2)}\longrightarrow0, \qquad \liminf_{m\to\infty}\|(-\Delta)^{\sigma/2}u_m\|_{L^p(B_1)}>0.
$$
The kernel is symmetric in the increment variable, uniformly elliptic, and satisfies the standard bounds on its first two derivatives in that variable. Its dependence on $x$ has a uniform modulus of continuity of order $(\log(1/r))^{-1/2}$ as $r\downarrow0$. Moreover, $\|Du_m\|_{L^p(B_4)}\to0$ for $\sigma>1$, and these norms remain bounded for $\sigma=1$. Thus the interior $L^p$ estimates for translation-invariant operators do not extend to this class under mere continuity in $x$.



\end{abstract}

\section{Introduction}
Let $n\geq1$ and $0<\sigma<2$. We consider nonlocal elliptic equations of the form
\begin{equation}
\inf_{\alpha\in\mathcal A}\frac12\int_{\mathbb R^n}\bigl(u(x+y)+u(x-y)-2u(x)\bigr)K_\alpha(x,y)\,dy=f(x) \quad\text{in }B_4,
\end{equation}
where $\mathcal A$ is an index set and $f\in L^p(B_4)$, $1\leq p<\infty$. We assume that $K_\alpha(x,y)=K_\alpha(x,-y)$ and, for some $0<\lambda\leq\Lambda$,
\begin{equation}
\frac{\lambda}{|y|^{n+\sigma}}\leq K_\alpha(x,y)\leq\frac{\Lambda}{|y|^{n+\sigma}},
\end{equation}
\begin{equation}
|D_y^iK_\alpha(x,y)|\leq\frac{\Lambda}{|y|^{n+\sigma+i}}, \qquad i=1,2,
\end{equation}
and
\begin{equation}
|K_\alpha(x,y)-K_\alpha(x',y)|\leq\frac{\Lambda\omega(|x-x'|)}{|y|^{n+\sigma}}
\end{equation}
for all $\alpha\in\mathcal A$, $x,x'\in B_4$, and $y\ne0$. Here $\omega:[0,\infty)\to[0,\infty)$ is continuous and nondecreasing, with $\omega(0)=0$. Recall that
\begin{equation}
(-\Delta)^{\sigma/2}u(x)=-\frac{c_{n,\sigma}}2\int_{\mathbb R^n}\frac{u(x+y)+u(x-y)-2u(x)}{|y|^{n+\sigma}}\,dy,
\end{equation}
where
$$
c_{n,\sigma}=\left(\int_{\mathbb R^n}\frac{1-\cos y_1}{|y|^{n+\sigma}}\,dy\right)^{-1}=2^{\sigma-1}\sigma\frac{\Gamma((n+\sigma)/2)}{\Gamma(1-\sigma/2)}\pi^{-n/2};
$$
see Fern\'andez-Real and Ros-Oton \cite{FernandezRealRosOton2024}.

Caffarelli and Silvestre \cites{CaffarelliSilvestre2009,CaffarelliSilvestre2011a,CaffarelliSilvestre2011b} established H\"older, $C^{1,\alpha}$, and Evans--Krylov estimates for fully nonlinear integro-differential equations, with estimates that recover the second-order theory as $\sigma\to2$. Further Schauder estimates were obtained by Jin and Xiong \cite{JinXiong2016} and Serra \cite{Serra2015}. For second-order linear elliptic equations, the Calder\'on--Zygmund estimate and a perturbation argument yield interior $W^{2,p}$ estimates under continuity of the coefficients, $1<p<\infty$; see Gilbarg and Trudinger \cite{GilbargTrudinger2001}. Caffarelli \cite{Caffarelli1989} developed the corresponding theory for fully nonlinear equations. We study whether analogous estimates for $(-\Delta)^{\sigma/2}u$ hold under (1.2)--(1.4).

For translation-invariant linear operators, Dong and Kim \cite{DongKim2012} proved $L^p$ estimates for $(-\Delta)^{\sigma/2}u$ for every $1<p<\infty$. Their kernels may be measurable in $y$, without homogeneity or symmetry, subject to ellipticity and a cancellation condition when $\sigma=1$. Their interior estimates include a gradient term for $\sigma\geq1$, with an arbitrarily small coefficient when $\sigma=1$; see (7.1)--(7.3) in \cite{DongKim2012}. See also Ba\~nuelos and Bogdan \cite{BanuelosBogdan2007} for related Fourier multiplier estimates and Zhang \cite{Zhang2013} for $L^p$-maximal regularity of nonlocal equations.

Dong and Kim \cite{DongKim2012}, in the discussion following Remark~2.4, asked whether their theory extends to $x$-dependent kernels with a uniform modulus of continuity in $x$. Mikulevi\v{c}ius and Pragarauskas \cite{MikuleviciusPragarauskas2014} obtained Sobolev estimates for variable kernels under H\"older continuity in $x$, with a restriction on $p$. Dong, Jung, and Kim \cite{DongJungKim2023} subsequently established $L^p$ estimates for every $1<p<\infty$ when the coefficient is H\"older continuous in $x$ with any positive exponent and merely measurable in $y$.

For more restricted classes of fully nonlinear operators, Yu \cite{Yu2017} proved $W^{\sigma,\epsilon}$ estimates, and Kitano \cite{Kitano2022} obtained $W^{\sigma,p}$ estimates under additional assumptions on the coefficients and the exponents. These results concern kernels with a quadratic angular dependence and homogeneity in the increment variable. For this class, dos Prazeres and Santos \cite{DosPrazeresSantos2024} also obtained optimal H\"older estimates for equations with $L^p$ right-hand sides.

In this paper, we show that (1.2)--(1.4) do not suffice for interior $L^p$ estimates for $(-\Delta)^{\sigma/2}u$, even for linear equations. Define
$$
Lu(x):=\frac12\int_{\mathbb R^n}\bigl(u(x+y)+u(x-y)-2u(x)\bigr)K(x,y)\,dy.
$$

\begin{theorem}
Let $1\leq p<\infty$ and $0<\sigma<2$. There exist a fixed kernel $K(x,y)=K(x,-y)$ satisfying (1.2)--(1.4), functions $u_m\in C^\infty(\mathbb R^n)\cap L^\infty(\mathbb R^n)$, and $f_m\in L^p(B_2)$ such that $Lu_m=f_m$ in $B_2$,
\begin{equation}
\|u_m\|_{L^\infty(\mathbb R^n)}\longrightarrow0,
\end{equation}
\begin{equation}
\|f_m\|_{L^p(B_2)}\longrightarrow0,
\end{equation}
\begin{equation}
\|Du_m\|_{L^p(B_4)}\longrightarrow0 \quad\text{if }\sigma>1,
\end{equation}
but
\begin{equation}
\liminf_{m\to\infty}\|(-\Delta)^{\sigma/2}u_m\|_{L^p(B_1)}>0.
\end{equation}
\end{theorem}

Theorem~1.1 rules out analogous interior estimates under (1.2)--(1.4) when $\sigma\ne1$. The case $\sigma=1$ follows from the additional bound $\sup_m\|Du_m\|_{L^p(B_4)}<\infty$; see Section~4. The kernel and its modulus of continuity are fixed throughout the sequence. The result concerns a priori estimates for smooth solutions; we do not construct a solution with $(-\Delta)^{\sigma/2}u\notin L^p(B_1)$.

The construction uses dyadic cosine modes and signed kernels supported in disjoint annuli on the Fourier side. Their coefficients tend to zero, which gives (1.4), but their squares have divergent sum. By grouping the frequencies into finite blocks, we arrange that the perturbation of $-(-\Delta)^{\sigma/2}$ cancels its contribution on one lower frequency in each block. A lacunary $L^p$ estimate controls the remaining terms, while the lower frequency gives (1.9).

In a related direction, Jiang and Xu \cite{JiangXu2026} constructed a fixed uniformly elliptic nonlocal operator for which integral Aleksandrov--Bakelman--Pucci estimates fail. Their coefficients are measurable in $x$. Here we retain the uniform continuity condition (1.4), as well as the derivative bounds (1.3).

In Section~2, we construct $K$ and verify (1.2)--(1.4). In Section~3, we define $u_m$ and compute $Lu_m$. Section~4 proves (1.6)--(1.9) and treats the case $\sigma=1$. Throughout the paper, $B_r$ denotes the ball of radius $r$ centered at the origin, and constants $C>0$ may change from line to line.

\section{Construction of the kernel \texorpdfstring{$K$}{K}}

We are going to construct a fixed kernel $K$ of the special form below:
$$
K(x,y):=\frac{c_{n,\sigma}}{|y|^{n+\sigma}}+\sum_{m=1}^\infty b_{\beta_m-2}(x)\sum_{j=\beta_m}^{\gamma_m}d_jb_j(x)q_j(y).
$$
where $d_j,b_j,q_j$ will be defined below.
\medskip

\noindent\textbf{Step 1: The kernels $q_j$.}
For $f\in\mathcal S(\mathbb R^n)$, we use the Fourier transform
\begin{equation}
\widehat f(\xi)=\int_{\mathbb R^n}e^{-iy\cdot\xi}f(y)\,dy
\end{equation}
and the inversion formula
\begin{equation}
f(x)=\frac1{(2\pi)^n}\int_{\mathbb R^n}e^{i\xi\cdot x}\widehat f(\xi)\,d\xi.
\end{equation}
Fix, once for each $n$, a nonnegative radial $\phi\in C_c^\infty(\mathbb R^n)$ with $\supp\phi\subset\{3/4<|\xi|<5/4\}$ and $\phi(e_1)=1$, where $e_1=(1,0,\ldots,0)$. Set $\widehat q(\xi)=|\xi|^\sigma\phi(\xi)$, and define $q$ by (2.2). The factor $|\xi|^\sigma$ has the homogeneity of the fractional Laplacian multiplier.

\begin{lemma}
The function $q\in\mathcal S(\mathbb R^n)$ is real and radial, and $\int_{\mathbb R^n}q(y)\,dy=0$.
\end{lemma}

\begin{proof}
Since $\widehat q\in C_c^\infty(\mathbb R^n)$, $q\in\mathcal S(\mathbb R^n)$. For every orthogonal matrix $O$,
$$
\int_{\mathbb R^n}e^{-iy\cdot\xi}q(y)\,dy=\widehat q(\xi)=\widehat q(O\xi)=\int_{\mathbb R^n}e^{-iz\cdot\xi}q(Oz)\,dz.
$$
The injectivity of the Fourier transform on $\mathcal{S}(\mathbb{R}^n)$ (see Stein and Shakarchi \cite{SteinShakarchi2003}) gives $q(Oz)=q(z)$. Since $\widehat q$ is real and even, (2.2) gives $q=\overline q$. Finally, $\int_{\mathbb R^n}q(y)\,dy=\widehat q(0)=0$.
\end{proof}

For $j\geq1$, define $\widehat q_j(\xi)=|\xi|^\sigma\phi(2^{-j}\xi)$. Then
\begin{equation}
\begin{aligned}
q_j(y)&=\frac1{(2\pi)^n}\int_{\mathbb R^n}e^{iy\cdot\xi}|\xi|^\sigma\phi(2^{-j}\xi)\,d\xi\\
&=\frac{2^{j(n+\sigma)}}{(2\pi)^n}\int_{\mathbb R^n}e^{iy\cdot(2^j\xi)}|\xi|^\sigma\phi(\xi)\,d\xi=2^{j(n+\sigma)}q(2^jy).
\end{aligned}
\end{equation}

\begin{lemma}
For $i=0,1,2$, there exists $C(i)<\infty$, depending only on $n,\sigma$, such that
\begin{equation}
\sum_{j=1}^\infty|D_y^iq_j(y)|\leq C(i)|y|^{-n-\sigma-i},\qquad y\ne0.
\end{equation}
\end{lemma}
\textbf{Remark.} For $i=0,1,2,$ $C(i)$ clearly depend on the choice of $\phi$. Since $\phi$ has been fixed, we omitted the dependence on $\phi$ in the lemma.
\begin{proof}
By (2.3), since $q\in\mathcal{S}(\mathbb{R}^n)$, for every $N>0$, $|D_y^iq_j(y)|\leq C_{i,N}2^{j(n+\sigma+i)}(1+2^j|y|)^{-N}$. Set $a=n+\sigma+i$ and fix $N=3a$. Splitting at $2^j|y|=1$ gives
\begin{equation}
\begin{aligned}
\sum_{j=1}^\infty|D_y^iq_j(y)|&=\sum_{2^j|y|\leq1}|D_y^iq_j(y)|+\sum_{2^j|y|>1}|D_y^iq_j(y)|\\
&\leq\sum_{2^j|y|\leq1}C_12^{ja-N}+\sum_{2^j|y|>1}C_22^{ja}(2^j|y|)^{-N}=:S_1+S_2,
\end{aligned}
\end{equation}
where $C_1,C_2$ absorb constants depending on $i,N$ and the fixed kernel $q$. If $\{j\geq1:2^j|y|\leq1\}\ne\varnothing$, let $J$ be its largest element. The geometric series give
$$
\begin{aligned}
S_1&\leq C_1\sum_{j=1}^J2^{ja}=C_12^{Ja}\frac{1-2^{-Ja}}{1-2^{-a}}\leq\frac{C_1|y|^{-a}}{1-2^{-a}}\leq C|y|^{-a},\\
S_2&=C_2|y|^{-a}(2^J|y|)^{a-N}\sum_{j=1}^\infty2^{j(a-N)}\leq C|y|^{-a},
\end{aligned}
$$
since $1/2<2^J|y|\leq1$. Otherwise, $|y|>1/2$ and
$$
\sum_{j=1}^\infty|D_y^iq_j(y)|\leq C|y|^{-N}\sum_{j=1}^\infty2^{j(a-N)}\leq C|y|^{-N}\leq C|y|^{-a}.
$$
This proves (2.4).
\end{proof}

\medskip
\noindent\textbf{Step 2: The coefficients.}
Let $b_j(x)=\cos(2^jx_1)$, $j\geq1$. We will use $2b_j^2=b_{j+1}+1$ in Section 3. Below in Section 3 we will also use $q_j$ as kernels in operators to act on $b_j$. Therefore, we now compute their action on $b_j$. Define
$$
L_ju(x):=\frac12\int_{\mathbb R^n}\bigl(u(x+y)+u(x-y)-2u(x)\bigr)q_j(y)\,dy.
$$
Since $\int_{\mathbb R^n}q_i(y)\,dy=0$,
\begin{equation}
\begin{aligned}
L_ib_j(x)&=\frac12\int_{\mathbb R^n}\bigl(\cos(2^j(x_1+y_1))+\cos(2^j(x_1-y_1))-2\cos(2^jx_1)\bigr)q_i(y)\,dy\\
&=\cos(2^jx_1)\int_{\mathbb R^n}\cos(2^jy_1)q_i(y)\,dy\\
&=\widehat q_i(2^je_1)b_j(x)=2^{j\sigma}\phi(2^{j-i}e_1)b_j(x).
\end{aligned}
\end{equation}
When $i\neq j$, $|2^{j-i}e_1|\geq2$ or $|2^{j-i}e_1|\leq1/2$. Therefore, $\supp\phi\subset\{3/4<|\xi|<5/4\}$ and $\phi(e_1)=1$ imply $L_ib_j=0$ for $i\ne j$, and $L_jb_j=2^{j\sigma}b_j$. We can see from the above calculation that operators $L_j$ are like coordinate projections acting on $b_j$. This is the key property in our construction.

Now we choose constants $d_j$. Let $d>0$ be chosen in Step~3 and set $d_j=d/\sqrt j$. Then $\sum_{j\geq1}d_j^2=\infty$ and $\sum_{j\geq1}d_j^4<\infty$. Starting with $\beta_1=3$, choose $\gamma_m\geq\beta_m$ to be the smallest integer such that
\begin{equation}
S_m:=\sum_{j=\beta_m}^{\gamma_m}d_j^2\geq m^2,
\end{equation}
and set $\beta_{m+1}=\gamma_m+3$. Thus the intervals $[\beta_m,\gamma_m]$ are pairwise disjoint, $\beta_m\to\infty$, and $\beta_m-2$ belongs to none of these intervals. Only $S_m\to\infty$ and the separation of these frequencies will be used; $m^2$ may be replaced by any sequence tending to infinity, and the gap $3$ by any larger integer.

Define
\begin{equation}
K(x,y):=\frac{c_{n,\sigma}}{|y|^{n+\sigma}}+\sum_{m=1}^\infty b_{\beta_m-2}(x)\sum_{j=\beta_m}^{\gamma_m}d_jb_j(x)q_j(y).
\end{equation}
The first term gives $-(-\Delta)^{\sigma/2}$ in $L$. The factors $b_{\beta_m-2}$ will produce the cancellation in Section~3.

\medskip
\noindent\textbf{Step 3: Verification of (1.2)--(1.4).}

\begin{lemma}
The series in (2.8) defines $K\in C^\infty(\mathbb R^n\times(\mathbb R^n\setminus\{0\}))$.
\end{lemma}

\begin{proof}
For multi-indices $\zeta,\delta$ and $\beta_m\leq j\leq\gamma_m$,
$$
\begin{aligned}
|\partial_x^\zeta(b_{\beta_m-2}b_j(x))|&\leq C_\zeta2^{j|\zeta|},\\
|\partial_y^\delta q_j(y)|&\leq C_{\delta,N}2^{j(n+\sigma+|\delta|)}(1+2^j|y|)^{-N}.
\end{aligned}
$$
Hence, for $|y|\geq\epsilon>0$,
$$
|\partial_x^\zeta\partial_y^\delta(b_{\beta_m-2}(x)b_j(x)q_j(y))|\leq C_{\zeta,\delta,N,\epsilon}2^{j(n+\sigma+|\zeta|+|\delta|-N)}.
$$
Choose $N>2(n+\sigma+|\zeta|+|\delta|)$. Since $d_j\leq d$ and the blocks are disjoint, the differentiated series converge locally uniformly for $y\ne0$.
\end{proof}

The radiality of $q_j$ gives $K(x,y)=K(x,-y)$. By $|b_j|\leq1$ and Lemma~2.2,
$$
\left|K(x,y)-\frac{c_{n,\sigma}}{|y|^{n+\sigma}}\right|\leq d\sum_{j=1}^\infty|q_j(y)|\leq\frac{dC(0)}{|y|^{n+\sigma}}.
$$
Fix $d$ by $dC(0)=c_{n,\sigma}/2$. Then
$$
\frac{c_{n,\sigma}}{2|y|^{n+\sigma}}\leq K(x,y)\leq\frac{3c_{n,\sigma}}{2|y|^{n+\sigma}}.
$$
For $i=1,2$, termwise differentiation and Lemma~2.2 give
\begin{equation}
\begin{aligned}
|D_y^iK(x,y)|&\leq\left|D_y^i\left(-\frac{c_{n,\sigma}}{|y|^{n+\sigma}}\right)\right|+d\sum_{j=1}^\infty|D_y^iq_j(y)|\\
&\leq\frac{A_i}{|y|^{n+\sigma+i}}.
\end{aligned}
\end{equation}
Thus (1.2)--(1.3) hold with $\lambda=c_{n,\sigma}/2$ and $\Lambda=\max\{3c_{n,\sigma}/2,A_1,A_2\}$.

For $j\in[\beta_m,\gamma_m]$ and $r=|x-x'|$, the mean value theorem and $|b_j|\leq1$ give
\begin{equation}
\begin{aligned}
&|d_jb_j(x)b_{\beta_m-2}(x)-d_jb_j(x')b_{\beta_m-2}(x')|\\
&\quad\leq\frac d{\sqrt j}\min\left\{r\sup_{t\in\mathbb R}\left|2^j\sin(2^jt)\cos(2^{\beta_m-2}t)+2^{\beta_m-2}\cos(2^jt)\sin(2^{\beta_m-2}t)\right|,2\right\}\\
&\quad\leq\frac d{\sqrt j}\min\{2^{j+1}r,2\}\leq\frac{2d}{\sqrt j}\min\{1,2^jr\}.
\end{aligned}
\end{equation}

\begin{lemma}
For $0<r\leq1/2$,
\begin{equation}
\sup_{j\geq1}\frac{\min\{1,2^jr\}}{\sqrt j}\leq\frac C{\sqrt{\log(1/r)}}.
\end{equation}
\end{lemma}

\begin{proof}
If $j\leq\frac12\log_2(1/r)$, then $2^jr\leq\sqrt r$ and
$$
\frac{\min\{1,2^jr\}}{\sqrt j}\leq\sqrt{\frac rj}\leq\sqrt r\leq\frac C{\sqrt{\log(1/r)}}.
$$
If $j>\frac12\log_2(1/r)$, the left-hand side is at most $j^{-1/2}\leq\sqrt{2\log2}/\sqrt{\log(1/r)}$.
\end{proof}

Combining (2.4), (2.10), and (2.11), for $0<|x-x'|\leq1/2$,
\begin{equation}
\begin{aligned}
|K(x,y)-K(x',y)|&\leq\frac{Cd}{\sqrt{\log(1/|x-x'|)}}\sum_{j\geq1}|q_j(y)|\\
&\leq\frac{Cd}{\sqrt{\log(1/|x-x'|)}\,|y|^{n+\sigma}}.
\end{aligned}
\end{equation}
For $|x-x'|\geq1/2$, (2.10) gives
\begin{equation}
|K(x,y)-K(x',y)|\leq Cd\sum_{j\geq1}\frac{|q_j(y)|}{\sqrt j}\leq\frac{Cd}{|y|^{n+\sigma}}.
\end{equation}
Consequently, (1.4) holds with
\begin{equation}
\omega(r)=
\begin{cases}
0,&r=0,\\
\displaystyle\frac{C_\omega}{\sqrt{\log(1/r)}},&0<r\leq1/2,\\
\displaystyle\frac{C_\omega}{\sqrt{\log2}},&r>1/2,
\end{cases}
\end{equation}
where $C_\omega$ is sufficiently large, depending on the fixed $n,\sigma,\phi,d,\lambda,\Lambda$. This $\omega$ is continuous and nondecreasing, with $\omega(0)=0$.

\section{Construction of \texorpdfstring{$u_m$}{u\_m}}

Define
$$
u_m(x)=2^{-(\beta_m-2)\sigma}b_{\beta_m-2}(x)+\frac2{S_m}\sum_{j=\beta_m}^{\gamma_m}d_j2^{-j\sigma}b_j(x).
$$
Then $u_m\in C^\infty(\mathbb R^n)\cap L^\infty(\mathbb R^n)$, and (2.6) gives
\begin{equation}
L_iu_m(x)=
\begin{cases}
\displaystyle\frac{2d_i}{S_m}b_i(x),&\beta_m\leq i\leq\gamma_m,\\
b_{\beta_m-2}(x),&i=\beta_m-2,\\
0,&\text{otherwise}.
\end{cases}
\end{equation}

We first justify termwise integration in $L$. For $u\in C^\infty(\mathbb R^n)$ with $u,Du,D^2u$ bounded,
\begin{equation}
\begin{aligned}
&\left|\sum_{m=1}^\infty b_{\beta_m-2}(x)\sum_{j=\beta_m}^{\gamma_m}d_jb_j(x)L_ju(x)\right|\\
&\quad\leq\sum_{m=1}^\infty\sum_{j=\beta_m}^{\gamma_m}\frac{d_j}2\int_{\mathbb R^n}|u(x+y)+u(x-y)-2u(x)|\,|q_j(y)|\,dy\\
&\quad\leq\frac d2\int_{\mathbb R^n}|u(x+y)+u(x-y)-2u(x)|\sum_{j=1}^\infty|q_j(y)|\,dy.
\end{aligned}
\end{equation}
By Lemma~2.2,
$$
\int_{\mathbb R^n}|u(x+y)+u(x-y)-2u(x)|\sum_{j=1}^\infty|q_j(y)|\,dy\leq C\int_{\mathbb R^n}\frac{\min\{|y|^2,1\}}{|y|^{n+\sigma}}\,dy<\infty.
$$
Tonelli's theorem therefore justifies the interchange. Also, by the definition of $c_{n,\sigma}$,
\begin{equation}
\begin{aligned}
&\frac{c_{n,\sigma}}2\int_{\mathbb R^n}\frac{b_j(x+y)+b_j(x-y)-2b_j(x)}{|y|^{n+\sigma}}\,dy\\
&\quad=c_{n,\sigma}b_j(x)\int_{\mathbb R^n}\frac{\cos(2^jy_1)-1}{|y|^{n+\sigma}}\,dy\\
&\quad=2^{j\sigma}c_{n,\sigma}b_j(x)\int_{\mathbb R^n}\frac{\cos y_1-1}{|y|^{n+\sigma}}\,dy=-2^{j\sigma}b_j(x).
\end{aligned}
\end{equation}

Since $\beta_m-2$ belongs to none of the intervals $\{[\beta_k,\gamma_k]\}_{k=1}^\infty$, (3.1), (3.3), and $2b_j^2=b_{j+1}+1$ yield
\begin{equation}
\begin{aligned}
f_m(x):=Lu_m(x)&=-b_{\beta_m-2}(x)-\frac2{S_m}\sum_{j=\beta_m}^{\gamma_m}d_jb_j(x)+\frac{b_{\beta_m-2}(x)}{S_m}\sum_{j=\beta_m}^{\gamma_m}2d_j^2b_j^2(x)\\
&=-b_{\beta_m-2}(x)-\frac2{S_m}\sum_{j=\beta_m}^{\gamma_m}d_jb_j(x)+\frac{b_{\beta_m-2}(x)}{S_m}\sum_{j=\beta_m}^{\gamma_m}d_j^2(b_{j+1}(x)+1)\\
&=-\frac2{S_m}\sum_{j=\beta_m}^{\gamma_m}d_jb_j(x)+\frac{b_{\beta_m-2}(x)}{S_m}\sum_{j=\beta_m}^{\gamma_m}d_j^2b_{j+1}(x).
\end{aligned}
\end{equation}
Here the constant terms in $2b_j^2$ cancel $-b_{\beta_m-2}$ because $\sum_{j=\beta_m}^{\gamma_m}d_j^2=S_m$.

\section{Proof of Theorem 1.1}
We verify (1.6)-(1.9) in turn to prove the main theorem. As $m\to\infty$, since $\beta_m\to\infty$ and $S_m\geq m^2$,
$$
\begin{aligned}
\|u_m\|_{L^\infty(\mathbb R^n)}&\leq2^{-(\beta_m-2)\sigma}+\frac2{S_m}\sum_{j=\beta_m}^{\gamma_m}d2^{-j\sigma}\\
&\leq2^{-(\beta_m-2)\sigma}+\frac{2d}{S_m}\frac{2^{-\beta_m\sigma}}{1-2^{-\sigma}}\longrightarrow0.
\end{aligned}
$$
This proves (1.6). To prove (1.7), we use the following partial sum estimate.

\begin{lemma}
For every $1\leq p<\infty$, there exists $C_{n,p}<\infty$ such that for all real $a_j$ and integers $\gamma\geq\beta\geq1$,
\begin{equation}
\left\|\sum_{j=\beta}^\gamma a_j\cos(2^jx_1)\right\|_{L^p(B_4)}\leq C_{n,p}\left(\sum_{j=\beta}^\gamma a_j^2\right)^{1/2}.
\end{equation}
\end{lemma}

\begin{proof}
Write $F(t)=\sum_{j=\beta}^\gamma a_j\cos(2^jt)$ and $A^2=\sum_{j=\beta}^\gamma a_j^2$. We claim that, for every integer $w\geq1$,
\begin{equation}
\int_0^{2\pi}|F(t)|^{2w}\,dt\leq2\pi(2w)!A^{2w}.
\end{equation}
We induct on the number of terms, simultaneously for all $w$. The one-term case follows from $|\cos t|\leq1$. For the induction step, let $p=2w$ and $G(t)=\sum_{j=\beta+1}^\gamma a_j\cos(2^jt)$. The shift $t\mapsto t+\pi/2^\beta$ leaves $G$ unchanged and reverses the sign of $\cos(2^\beta t)$. Hence
\begin{equation}
\begin{aligned}
\int_0^{2\pi}|F(t)|^p\,dt&=\frac12\int_0^{2\pi}|F(t)|^p\,dt+\frac12\int_0^{2\pi}\left|F\left(t+\frac\pi{2^\beta}\right)\right|^p\,dt\\
&=\frac12\int_0^{2\pi}\bigl((G(t)+a_\beta\cos(2^\beta t))^p+(G(t)-a_\beta\cos(2^\beta t))^p\bigr)\,dt\\
&=\int_0^{2\pi}\sum_{j=0}^w\binom p{2j}G(t)^{2j}(a_\beta\cos(2^\beta t))^{p-2j}\,dt\\
&\leq\sum_{j=0}^w\binom p{2j}a_\beta^{p-2j}\int_0^{2\pi}G(t)^{2j}\,dt.
\end{aligned}
\end{equation}
Let $B^2=\sum_{j=\beta+1}^\gamma a_j^2$. By the induction hypothesis,
\begin{equation}
\begin{aligned}
\int_0^{2\pi}|F(t)|^p\,dt&\leq2\pi a_\beta^p+2\pi\sum_{j=1}^w\binom p{2j}(2j)!a_\beta^{p-2j}B^{2j}\\
&=2\pi p!\sum_{j=0}^w\frac{a_\beta^{p-2j}B^{2j}}{(p-2j)!}\\
&\leq2\pi p!\sum_{j=0}^w\frac{w!}{j!(w-j)!}a_\beta^{p-2j}B^{2j}\\
&=2\pi p!(a_\beta^2+B^2)^w=2\pi p!A^{2w}.
\end{aligned}
\end{equation}
Here $w!(2w-2j)!\geq j!(w-j)!$ for $0\leq j\leq w$. This proves (4.2).

For arbitrary $1\leq p<\infty$, choose an integer $w$ with $p\leq2w<p+3$. H\"older's inequality and (4.2) give
\begin{equation}
\begin{aligned}
\int_0^{2\pi}|F(t)|^p\,dt&\leq(2\pi)^{(2w-p)/(2w)}\left(\int_0^{2\pi}|F(t)|^{2w}\,dt\right)^{p/(2w)}\\
&\leq(2\pi)^{(2w-p)/(2w)}\bigl(2\pi(2w)!A^{2w}\bigr)^{p/(2w)}=2\pi((2w)!)^{p/(2w)}A^p.
\end{aligned}
\end{equation}
By periodicity and Fubini's theorem,
$$
\int_{B_4}|F(x_1)|^p\,dx\leq\int_{[-4,4]^{n-1}}\left(2\int_{-\pi}^{\pi}|F(t)|^p\,dt\right)\,dx_2\cdots dx_n=C_0^p\int_{-\pi}^{\pi}|F(t)|^p\,dt.
$$
Together with (4.5), this gives (4.1), with $C_{n,p}=C_0(2\pi)^{1/p}((2w)!)^{1/(2w)}$.
\end{proof}

By (3.4), $|b_{\beta_m-2}|\leq1$, and Lemma~4.1,
$$
\begin{aligned}
\|f_m\|_{L^p(B_2)}&\leq\frac{2C_{n,p}}{S_m}\left(\sum_{j=\beta_m}^{\gamma_m}d_j^2\right)^{1/2}+\frac{C_{n,p}}{S_m}\left(\sum_{j=\beta_m}^{\gamma_m}d_j^4\right)^{1/2}\\
&\leq\frac{2C_{n,p}}{\sqrt{S_m}}+\frac C{S_m}\longrightarrow0.(m\to \infty)
\end{aligned}
$$
In the last inequality we applied $\sum_{j\geq1}d_j^4=d^4\sum_{j\geq1}j^{-2}<\infty$. This proves (1.7).

For $\sigma>1$, as $m\to \infty$, differentiation and $S_m\geq m^2$ gives
$$
\begin{aligned}
\|Du_m\|_{L^\infty(B_4)}&\leq C\left\|2^{(\beta_m-2)(1-\sigma)}\sin(2^{\beta_m-2}t)+\frac2{S_m}\sum_{j=\beta_m}^{\gamma_m}d_j2^{(1-\sigma)j}\sin(2^jt)\right\|_{L^\infty([-4,4])}\\
&\leq C\left(2^{(\beta_m-2)(1-\sigma)}+\frac{2d\,2^{(1-\sigma)\beta_m}}{S_m(1-2^{1-\sigma})}\right)\longrightarrow0.
\end{aligned}
$$
Thus $\|Du_m\|_{L^p(B_4)}\leq|B_4|^{1/p}\|Du_m\|_{L^\infty(B_4)}\to0$, proving (1.8).

To prove (1.9), (3.3) and Lemma~4.1 give
\begin{equation}
\|(-\Delta)^{\sigma/2}u_m-b_{\beta_m-2}\|_{L^p(B_1)}=\left\|\frac2{S_m}\sum_{j=\beta_m}^{\gamma_m}d_jb_j\right\|_{L^p(B_1)}\leq\frac{2C_{n,p}}{\sqrt{S_m}}.
\end{equation}
On the other hand, as $m\to \infty$, the Riemann--Lebesgue lemma yields
$$
\int_{B_1}|b_{\beta_m-2}(x)|^2\,dx=\frac12\int_{B_1}\bigl(1+\cos(2^{\beta_m-1}x_1)\bigr)\,dx\longrightarrow\frac12|B_1|.
$$
If $p\geq2$, H\"older's inequality gives $\|b_{\beta_m-2}\|_{L^2(B_1)}^2\leq\|b_{\beta_m-2}\|_{L^p(B_1)}^2|B_1|^{(p-2)/p}$, whence
\begin{equation}
\|b_{\beta_m-2}\|_{L^p(B_1)}\geq C|B_1|^{1/p}>0
\end{equation}
for all sufficiently large $m$. Together with (4.6), this implies
\begin{equation}
\liminf_{m\to\infty}\norm{(-\Delta)^{\sigma/2}u_m}_{L^p(B_1)}\geq C|B_1|^{1/p}>0.
\end{equation}
If $1\leq p<2$, then $|b_{\beta_m-2}|^p\geq|b_{\beta_m-2}|^2$, so
$$
\|b_{\beta_m-2}\|_{L^p(B_1)}\geq\left(\int_{B_1}b_{\beta_m-2}^2(x)\,dx\right)^{1/p}=\|b_{\beta_m-2}\|_{L^2(B_1)}^{2/p}.
$$
This lower bound and (4.6) again give (1.9), completing the proof of Theorem~1.1.

It remains to rule out the corresponding interior estimate when $\sigma=1$. In this case,
$$
(u_m)_{x_1}=-\sin(2^{\beta_m-2}x_1)-\frac2{S_m}\sum_{j=\beta_m}^{\gamma_m}d_j\sin(2^jx_1).
$$
The proof of Lemma~4.1 also applies to sine sums: $t\mapsto t+\pi/2^\beta$ reverses the sign of $\sin(2^\beta t)$ and leaves the terms with $j\geq\beta+1$ unchanged. Consequently,
$$
\begin{aligned}
\left\|\sum_{j=\beta_m}^{\gamma_m}d_j\sin(2^jx_1)\right\|_{L^p(B_4)}&\leq C'_{n,p}\left(\sum_{j=\beta_m}^{\gamma_m}d_j^2\right)^{1/2}=C'_{n,p}\sqrt{S_m},\\
\|Du_m\|_{L^p(B_4)}&\leq\|\sin(2^{\beta_m-2}x_1)\|_{L^p(B_4)}+\frac{2C'_{n,p}}{\sqrt{S_m}}\leq|B_4|^{1/p}+\frac{2C'_{n,p}}{\sqrt{S_m}}.
\end{aligned}
$$
Since $S_m\geq m^2$, $\sup_m\|Du_m\|_{L^p(B_4)}\leq M$ for some $M>0$. Let 
\[c=\liminf_{m\to\infty}\|(-\Delta)^{1/2}u_m\|_{L^p(B_1)}>0\] 
and choose $\epsilon=\min\{1,c/(2M)\}$, decreasing it if necessary so that $\epsilon<1$. If the estimate held under (1.2)--(1.4), then (1.6)--(1.7) would give $c\leq\epsilon M\leq c/2$, a contradiction.

\medskip
\noindent\textbf{Acknowledgements}: I am sincerely grateful to Professor Jingang Xiong, my advisor, and to Professor Tianling Jin for suggesting this research problem and for their valuable guidance and assistance in revising the manuscript.

\medskip
\noindent\textbf{AI usage statement}: During the preparation of this work, Chat GPT 5.6 Sol was used to assist in calculating. The manuscript was written by the author. The author takes full responsibility for the content of this paper.

\bigskip
\bigskip

\noindent C. Liu

\noindent School of Mathematical Sciences, Beijing Normal University, Beijing 100875, China\\
Email: \textsf{202531130999@mail.bnu.edu.cn}

\bigskip


\begin{bibdiv}
\begin{biblist}[\small]

\bib{BanuelosBogdan2007}{article}{author={Ba\~nuelos, Rodrigo}, author={Bogdan, Krzysztof}, title={L\'evy processes and Fourier multipliers}, journal={J. Funct. Anal.}, volume={250}, date={2007}, number={1}, pages={197--213},}

\bib{Caffarelli1989}{article}{author={Caffarelli, Luis A.}, title={Interior a priori estimates for solutions of fully nonlinear equations}, journal={Ann. of Math. (2)}, volume={130}, date={1989}, number={1}, pages={189--213},}

\bib{CaffarelliSilvestre2009}{article}{author={Caffarelli, Luis A.}, author={Silvestre, Luis}, title={Regularity theory for fully nonlinear integro-differential equations}, journal={Comm. Pure Appl. Math.}, volume={62}, date={2009}, number={5}, pages={597--638},}

\bib{CaffarelliSilvestre2011a}{article}{author={Caffarelli, Luis A.}, author={Silvestre, Luis}, title={The Evans--Krylov theorem for nonlocal fully nonlinear equations}, journal={Ann. of Math. (2)}, volume={174}, date={2011}, number={2}, pages={1163--1187},}

\bib{CaffarelliSilvestre2011b}{article}{author={Caffarelli, Luis A.}, author={Silvestre, Luis}, title={Regularity results for nonlocal equations by approximation}, journal={Arch. Ration. Mech. Anal.}, volume={200}, date={2011}, number={1}, pages={59--88},}

\bib{DongJungKim2023}{article}{author={Dong, Hongjie}, author={Jung, Pilgyu}, author={Kim, Doyoon}, title={Boundedness of non-local operators with spatially dependent coefficients and $L_p$-estimates for non-local equations}, journal={Calc. Var. Partial Differential Equations}, volume={62}, date={2023}, pages={Paper No. 62},}

\bib{DongKim2012}{article}{author={Dong, Hongjie}, author={Kim, Doyoon}, title={On $L_p$-estimates for a class of non-local elliptic equations}, journal={J. Funct. Anal.}, volume={262}, date={2012}, number={3}, pages={1166--1199},}

\bib{FernandezRealRosOton2024}{book}{author={Fern\'andez-Real, Xavier}, author={Ros-Oton, Xavier}, title={Integro-differential elliptic equations}, series={Progress in Mathematics}, volume={350}, publisher={Birkh\"auser/Springer}, address={Cham}, date={2024},}

\bib{GilbargTrudinger2001}{book}{author={Gilbarg, David}, author={Trudinger, Neil S.}, title={Elliptic partial differential equations of second order}, series={Classics in Mathematics}, publisher={Springer-Verlag}, address={Berlin}, date={2001},}

\bib{JiangXu2026}{article}{author={Jiang, Yongsheng}, author={Xu, Lei}, title={A counterexample to integral Aleksandrov--Bakelman--Pucci estimates for a nonlocal operator}, journal={arXiv:2610.16162}, date={2026}, url={https://arxiv.org/abs/2610.16162},}

\bib{JinXiong2016}{article}{author={Jin, Tianling}, author={Xiong, Jingang}, title={Schauder estimates for nonlocal fully nonlinear equations}, journal={Ann. Inst. H. Poincar\'e Anal. Non Lin\'eaire}, volume={33}, date={2016}, number={5}, pages={1375--1407},}

\bib{Kitano2022}{article}{author={Kitano, Shuhei}, title={$W^{\sigma,p}$ a priori estimates for fully nonlinear integro-differential equations}, journal={Calc. Var. Partial Differential Equations}, volume={61}, date={2022}, number={4}, pages={Paper No. 153},}

\bib{MikuleviciusPragarauskas2014}{article}{author={Mikulevi\v{c}ius, Remigijus}, author={Pragarauskas, Henrikas}, title={On the Cauchy problem for integro-differential operators in Sobolev classes and the martingale problem}, journal={J. Differential Equations}, volume={256}, date={2014}, number={4}, pages={1581--1626},}

\bib{DosPrazeresSantos2024}{article}{author={dos Prazeres, Disson S.}, author={Santos, Makson S.}, title={Optimal regularity for fully nonlinear nonlocal equations with unbounded source terms}, journal={arXiv:2409.03216}, date={2024}, url={https://arxiv.org/abs/2409.03216},}

\bib{Serra2015}{article}{author={Serra, Joaquim}, title={$C^{\sigma+\alpha}$ regularity for concave nonlocal fully nonlinear elliptic equations with rough kernels}, journal={Calc. Var. Partial Differential Equations}, volume={54}, date={2015}, number={4}, pages={3571--3601},}

\bib{SteinShakarchi2003}{book}{author={Stein, Elias M.}, author={Shakarchi, Rami}, title={Fourier analysis: An introduction}, series={Princeton Lectures in Analysis}, volume={1}, publisher={Princeton University Press}, address={Princeton, NJ}, date={2003},}

\bib{Yu2017}{article}{author={Yu, Hui}, title={$W^{\sigma,\epsilon}$-estimates for nonlocal elliptic equations}, journal={Ann. Inst. H. Poincar\'e Anal. Non Lin\'eaire}, volume={34}, date={2017}, number={5}, pages={1141--1153},}

\bib{Zhang2013}{article}{author={Zhang, Xicheng}, title={$L^p$-maximal regularity of nonlocal parabolic equations and applications}, journal={Ann. Inst. H. Poincar\'e Anal. Non Lin\'eaire}, volume={30}, date={2013}, number={4}, pages={573--614},}

\end{biblist}
\end{bibdiv}
\end{document}